\documentclass[]{article}
\usepackage{amsmath}
\usepackage{amsfonts}
\usepackage{graphicx}
\usepackage{amsthm}
\usepackage{mathtools}
\usepackage{hyperref}
\usepackage{amssymb}
\usepackage{xcolor}

\theoremstyle{plain}

\newtheorem{prop}{Proposition}[section]
\newtheorem{thm}{Theorem}
\newtheorem{defn}{Definition}[section]

\newtheorem*{thm*}{Theorem}
\newtheorem{cor}{Corollary}[section]

\newtheorem{rem}{Remark}[section]
\title{Margulis-Soifer theorem for one-relator groups}
\author{Azer Akhmedov \footnote{Azer Akhmedov, Department of Mathematics,
North Dakota State University,
Fargo, ND, 58102, USA. E-mail: azer.akhmedov@ndsu.edu}, Atabay Mahmudov\footnote{Atabay Mahmudov, Department of Mathematics,
North Dakota State University,
Fargo, ND, 58102, USA. E-mail: atabay.mahmudov@ndsu.edu}}
\date{September 22, 2026}

\begin{document}

	\maketitle
	\begin{abstract}
		We establish the Margulis–Soifer dichotomy for one-relator groups: every one-relator group is either virtually solvable or has a maximal subgroup of infinite index. We also present examples of one-relator groups with and without free maximal subgroups of infinite index, as well as examples that possess both free and non-free infinite index maximal subgroups. Triviality of the Frattini subgroup is also shown for all non-solvable one-relator groups. We close the paper with a short list of questions.
	\end{abstract}

	\section{Introduction} 
    The Margulis-Soifer theorem is an important dichotomy for linear groups. It states that a finitely generated linear group is either virtually solvable or contains a maximal subgroup of infinite index, and the two cases are mutually exclusive $\cite{ms1, ms2, margulis-soifer}$. In addition to linear groups, this dichotomy has been shown to hold, sometimes with a minor modification, for numerous other classes of groups; we refer the reader to the survey $\cite{ggs}$. 

 \medskip 
 
    For $PL_{+}(I)$, the group of piecewise linear orientation preserving homeomorphisms of the interval $I = [0,1]$, the dichotomy fails. In a forthcoming paper \cite{am}, the authors state and prove a right analog of this dichotomy in the class $PL_{+}(I)$ as well as examine it for various classes in $\mathrm{Homeo}_{+}(I)$. 

    \medskip 
    
     In this paper, we obtain a positive result for one-relator groups.

     \begin{thm} \label{thm: main} A one-relator group contains a maximal subgroup of infinite index if and only if it is not virtually solvable. 
     \end{thm}

     Recall that a one-relator group is virtually solvable if and only if it is solvable (this interesting property is also shared by the groups in $PL_{+}(I)$, see \cite{bleak}). Non-solvable one-relator groups contain a copy of a non-Abelian free group, so the Tits Alternative holds in the class of one-relator groups.  On the other hand,  solvable non-trivial one-relator groups consist of the following short list up to an isomorphism: $$\mathbb{Z}, \ \mathbb{Z}/n\mathbb{Z} \  \mathrm{for} \ n\geq 2, \ \mathbb{Z}^2, \ \pi _1(\mathrm{Klein \ bottle}) \cong \langle a, b \ | \ aba^{-1} = b^{-1}\rangle \  \mathrm{and \ the}$$  $$\mathrm{\ Baumslag-Solitar \ groups} \  BS(1,n) = \langle a, b \ | \ aba^{-1} = b^n\rangle , n\in \mathbb{Z}, |n| \geq 2.$$ All these groups are linear, so by Margulis-Soifer theorem every maximal subgroup is necessarily of finite index.  This fact can also be verified directly; it is a small exercise for $BS(1,n)$ and trivial for the other groups in the list. 

\medskip 

     For an $n$-generated one-relator group, we treat the cases of $n\geq 3$ and $n=2$ separately. In the former case, we use the largeness of these groups and observe an important property of large groups, see Proposition $\ref{prop:quotientoflarge}$.  As a byproduct, we also show that finitely generated large groups have infinite girth, thus answering the question in $\cite{azer2026}$. In a recent work \cite{abgm}, the authors obtain the same result (about girth) independently with a similar method. In the latter case, the proof is more involved; we use a characterization of Minasyan-Osin which shows that a 2-generated one-relator group belongs to at least one of the three types, namely, it is either acylindrically hyperbolic, or it is a special type of an HNN extension (possibly infinite cyclic), or it is a mapping torus of a free group (see Proposition \ref{prop:osin} for a precise statement). The first case lends itself to an existing result in the literature, but in the second and third cases, we divide them further in sub-cases so that in each variation we are able to establish the existence of a maximal subgroup of infinite index.

     \medskip 

      In the proof of Margulis-Soifer Theorem  \cite{ms2}, for a non-virtually solvable linear group, a proper profinitely dense free subgroup $F$ is produced. This free subgroup $F$ cannot be contained in any proper finite index subgroup. Then, by Zorn's lemma, $F$ is contained in a maximal subgroup $M$ of infinite index. Here, because of the use of Zorn's lemma, one cannot control the algebraic properties of $M$. In particular, freeness (when passing from $F$ to $M$) can easily be lost. One of the very interesting questions about infinite-index maximal subgroups is whether one can expect them to be finitely generated. This question ({\bf the Finite Generation Problem}) is still open for the groups $SL(n,\mathbb{Z}), n\geq 3$. It is also interesting in the context of one-relator groups.  We do not know a single example of a non-solvable one-relator group which possesses a finitely generated infinite index maximal subgroup. Regarding freeness, it is unclear for which linear groups one can hope for the existence of a free maximal subgroup of infinite index ({\bf the Freeness Problem}). This question too is open for the groups $PSL(n,\mathbb{Z}), n\geq 3$ (see \cite{gm}) Perhaps even more interestingly, it is unclear how to identify linear groups where all infinite index maximal subgroups are free. Let us point out that, by standard covering theory and by Ahlfors' theorem \footnote{This theorem states that the fundamental group of a connected non-compact surface (i.e. 2-manifold) is free.}, for a closed connected surface $\Sigma $ (other than $\mathbb{S}^2$ and $\mathbb{RP}^2$), in the fundamental group $\pi _1(\Sigma )$, 1) all infinite index maximal subgroups are free, and free maximal subgroups are of infinite index; 2) any finitely generated maximal subgroup has a finite index. (for the second claim, we also need to invoke Scott's subgroup separability theorem \cite{scott}).  

      \medskip 
      
      Another interesting question is if, in a given group, infinite index maximal subgroups are abundant; for example, if there exist infinite index maximal subgroups of different nature or if there are $2^{\aleph _0}$ infinite index maximal subgroups ({\bf the Variety Problem}). This problem has been highlighted in the survey \cite{ggs} for linear groups.  

\medskip 

       We study free maximal subgroups in the class of one-relator groups with the Freeness Problem and the Variety Problem in mind. We obtain the following theorem. 

      \begin{thm} \label{thm:main2} a) If $\Gamma $ is a non-cyclic one-relator group that has torsion or is 2-free, then $\Gamma $ has a free maximal subgroup of infinite index.

b) There exists a one-relator group $\Gamma $ with torsion that possesses both a free infinite index maximal subgroup and a non-free infinite index maximal subgroup. 

c) There exists a torsion-free one-relator group that possesses both a free infinite index maximal subgroup and a non-free infinite index maximal subgroup.

      \end{thm}
      
       Let us point out that if a group is non-cyclic and has an infinite center, then its maximal subgroup cannot be non-Abelian free. Therefore, the one-relator groups such as  $\langle x, y \ | \ x^p = y^q \rangle , p, q\geq 2, (p,q) = 1$ (the torus knot groups), $BS(n,n) = \langle a, b \ | \ ab^na^{-1}=b^n\rangle , n\geq 1$ will have no free maximal subgroup. Among center-free one-relator groups, we also have examples with a non-free maximal subgroup, such as $\langle a, b, c \ | \ c^2 = 1\rangle $ and $BS(n,-n) = \langle a, b \ | \ ab^na^{-1}=b^{-n}\rangle , n\geq 2$ (see Section 4). Regarding the remaining non-solvable Baumslag-Solitar groups $BS(p,q), |p|, |q|\geq 2, |p|\neq |q|$, we will observe that they admit free maximal subgroups. 

      \medskip 

      The richness of a group with maximal subgroups may manifest itself in the smallness of its Frattini subgroup. This has been shown by Platonov \cite{platonov} and Wehrfritz \cite{weh} for finitely generated linear groups \cite{platonov} , by Ivanov  for finitely generated subgroups of mapping class groups \cite{iva87}, and by I.Kapovich for non-elementary word hyperbolic groups \cite{ka2001}. Gelander and Glasner \cite{ggc} extend these  results to countable  (not necessarily finitely generated) setting by a unified method (in the third case, proving finiteness of the Frattini subgroup for all non-elementary convergence groups). The smallness of the Frattini subgroup can also be viewed as part of the Variety Problem. We show the following theorem for one-relator groups.

    \begin{thm} \label{thm:main3} The Frattini subgroup of a non-solvable one-relator group is trivial.
    \end{thm}

     $\textbf{Acknowledgment:}$ The first author thanks T.Amrutam for informing about \cite{abgm}.  The second author thanks J. P. Mutanguha for a helpful discussion and suggestions.

      {\bf AI Disclosure:} We have used GPT-5.6 Luna for literature search and clarification of the results from the literature.

     \section{The case of $n \geq 3$ generators}

	For $n$-generated one-relator groups with $n \geq 3$, we use the well-known Baumslag-Pride theorem which states that a group of deficiency at least 2 is a large group. We demonstrate that finitely generated large groups have a non-virtually-solvable linear quotient. As a by-product of this observation, we also show that finitely generated large groups have infinite girth, thus answering the question in $\cite{azer2026}$. In a recent work \cite{abgm}, the authors obtain the same result independently with a very similar method.

	\begin{prop} \label{prop:quotientoflarge}
		Let $G$ be a finitely generated large group. Then it has a non-virtually-solvable linear quotient.
	\end{prop}
	\begin{proof}
		Let $H$ be a finite-index subgroup of $G$ surjecting onto the free group $F_2$ on two generators. Now let $K = \text{Core}_G(H) \unlhd G$. Since $H$ has finite index in $G$, we have that $K$ also has finite index in $G$, say $[G:K]=d$. Now let $\pi :H \to F_2$ be the surjection onto $F_2$. The image of $K$ has finite index since 
		\[
		[F_2 : \pi(K)] = [\pi(H): \pi(K)] \leq [H:K]<\infty.
		\]
		
		Then the image of $K$ in $F_2$ is a free group of rank at least 2. Indeed, by the Nielsen-Schreier formula, 
		\[
		\text{rank}(\pi(K)) = 1 + [F_2 : \pi(K)] (2-1) = 1 + [F_2 : \pi(K)] \geq 2.
		\]
		Thus, $K$ surjects onto a free group $F_r$, $r \geq 2$. Let $\theta: K \to F_r$ be that surjection. Let $N = \text{ker}(\theta)$ and $N_0 = \text{Core}_G(N) = \bigcap_{g \in G} g N g^{-1} = \bigcap_{g \in G} N^{g}$. Since $K$ has finite index in $G$ and $N \unlhd K$, we can make $N_0$ a finite intersection. \\

		Now set $Q : = G/N_0$. We show that $Q$ is virtually linear, hence linear. 
		The group $K/N_0$ has finite index in $Q$ since $K$ has finite index in $G$. Now notice that $K/N_0$ embeds into the finite product $\prod_{i=1}^{d} K/N^{g_i}$. Indeed, if you consider the homomorphism 
		\[
		f: K \to \prod_{i=1}^{d} K/N^{g_i} \; \; \text{defined by} \; \; f(k) = (kN^{g_1},...,kN^{g_d})
		\]
		the kernel is precisely $N_0$ and thus $K/N_0 \hookrightarrow \prod_{i=1}^{d} K/N^{g_i}$. \\
		Since $K \unlhd G$, conjugation by $g_i$ induces an isomorphism $K/N^{g_i} \cong K/N \cong F_r$. Therefore, we see that $K/N_0$ embeds into a finite product of free groups and hence is a linear group. Thus, $Q$ is a virtually linear group. However, a virtually linear group is linear. Moreover, since $N_0 \leq N \leq K$, $K/N_0$ surjects onto $K/N$ which is a free group. Hence, $K/N_0$ cannot be virtually solvable. Since $K/N_0$ has finite index in $Q$ and is not virtually solvable, $Q$ cannot be virtually solvable either. Hence, we have a non-virtually-solvable finitely generated linear quotient $Q$ of the \textit{large group} $G$.
	\end{proof}
	
	\begin{cor}
		Let $G$ be an $n$-generated one-relator group where $n \geq 3$. Then $G$ has a maximal subgroup of infinite index. 
	\end{cor}
	
	\begin{proof}
		Let $q : G \to Q$ be the canonical surjection onto the quotient $Q$ from the Proposition 2.1. Since $Q$ is non-virtually solvable linear group, it has a maximal subgroup of infinite index by Margulis-Soifer theorem (see $\cite{margulis-soifer}$). Let us denote a maximal subgroup of $Q$ with infinite index by $M$. Then the pullback $\tilde{M} : = q^{-1}(M)$ of $M$ is a maximal subgroup of $G$ with infinite index. Note that $\text{ker}(q) \leq q^{-1}(M) \leq L$. Suppose $q^{-1}(M) < L < G$ for some $L < G$. Notice that if $q(L) = q(G)$ then for any $g \in G$, there exists $\ell \in L$ such that $q(g) =q(\ell)$ and thus $g \in L$ since $g \ell^{-1} \in \text{ker}(q) \leq L$. Hence, applying the quotient map $q$, we obtain that $M < q(L) < Q$ which contradicts the maximality of $M$. \\ Moreover, the map $G/ \tilde{M} \to Q/M$ given by $g \tilde{M} \mapsto q(g) M$ is a bijection. Hence, $[G: \tilde{M}]= [Q:M] = \infty$.
	\end{proof}
	
	Now as another corollary, we give a short answer to the question posed 	 in $\cite{azer2026}$
	\begin{cor} \label{cor:girth} 
		Suppose $G$ is a finitely generated large group. Then $\text{girth}(G) = \infty$. 
	\end{cor}

	\begin{proof}
		Akhmedov proved in $\cite{azer2005}$ that girth alternative holds for linear groups. Hence, the group $Q$ as in Proposition $2.1$ has infinite girth. Moreover, Proposition 1.1 of the same paper states that if $G$ is a finitely generated group with $N \unlhd G$ such that $\text{girth}(G/N)= \infty$ and $G/N \ncong \mathbb{Z}$ , then $\text{girth}(G)= \infty$. Since the group $Q$ is non-virtually-solvable, we have that $\text{girth}(G)= \infty$. 
	\end{proof}

     \section{ The case of two generators}

	To prove that every $2$-generated one-relator group is either virtually solvable or has a maximal subgroup of infinite index, we refer to the following characterization of the 2-generated one-relator groups due to A. Minasyan and D. Osin \cite{minos}. 
	\begin{prop} \label{prop:osin} 
		Let $G$ be a group with two generators and one defining relator. Then at least
		one of the following holds:
		\begin{enumerate}
			\item[(i)] $G$ is acylindrically hyperbolic;
			
			\item[(ii)] $G$ contains an infinite cyclic $s$-normal subgroup. More precisely,
			either $G$ is infinite cyclic or it is an HNN-extension of the form
			\[
			G = \langle a,b,t \mid a^t = b, r=1\rangle
			\]
			of a one-relator group
			\[
			H = \langle a,b \mid r\rangle
			\]
			with non-trivial center, so that $a^k=b^l$ in $H$ for some
			$k,l \in \mathbb{Z}\setminus\{0\}$. In the latter case $H$ is
			finitely generated free-by-cyclic and contains a finite index normal
			subgroup splitting as a direct product of a finitely generated free group
			with an infinite cyclic group;
			
			\item[(iii)] $G$ is isomorphic to the mapping torus of an injective endomorphism
			of a finitely generated free group.
		\end{enumerate}
		
		Moreover, the possibilities \textup{(i)} and \textup{(ii)} are mutually
		exclusive.
	\end{prop}

    Now we remind a reader the notion of a profinitely dense subgroup. 
	\begin{defn}
		Let $G$ be a group and $H \leq G$. We say that $H$ is a profinitely dense in $G$ if for every homomorphism $\phi : G \to Q$ to a finite group $Q$, we have 
		\[
		\phi(H) = \phi(G).
		\]
	\end{defn}
	
    Profinitely dense subgroups have been used crucially in the original work of Margulis and Soifer \cite{ms2} as a tool to produce maximal subgroups of infinite index. By Zorn's lemma, we know that every proper subgroup of a finitely generated group is contained in a maximal subgroup. If a proper subgroup is also profinitely dense, then a maximal subgroup containing it must necessarily be of infinite index. For convenience of the reader, we include the following well-known proposition from \cite{ms2}.

	\begin{prop} \label{prop:profinitedense}
		Let $D$ be a proper profinitely dense subgroup of a finitely generated group $G$. Then every maximal subgroup containing $D$ has infinite index.
	\end{prop}
	
	\begin{proof}
		Let $M$ be a maximal subgroup containing $D$. Suppose that $[G:M]=m$. Then $G$ acts on the set of left cosets of $M$ by the multiplication and induces the homomorphism $\rho :G \to \text{Sym}(G/M)$. Now let $C: = \text{Core}_G(M) = \ker(\rho)$. Hence, we have $G/C \cong \rho(G) \leq S_m$ and thus $G /C$ is a finite group. Now consider the natural projection $q :G \to G/C$. As $G/C$ is a finite group, we have $q(D) = q(G)$ by the profinite density of $D$. Then we have 
		\[
		q(D) \leq q(M) = M/C < G/C = q(G)
		\]
		which contradicts the fact that $q(D) = q(G)$. 
		
	\end{proof}

	\begin{prop}
		Every 2-generated one-relator group is either virtually solvable or has a maximal subgroup of infinite index.  
	\end{prop}
	
	\begin{proof}
		To prove the proposition, we will show that in all cases of the Proposition \ref{prop:osin}, either the group is virtually-solvable or has a maximal subgroup of infinite index. 
		
		\begin{enumerate}
			\item[(i)] Every countable acylindrically hyperbolic group admits a highly transitive action on an infinite set with a finite kernel due to Osin and Hull in $\cite{hull-osin}$. 
            In particular, every countable acylindrically hyperbolic group admits a 2-transitive action, that is, for any two ordered pairs of distinct points $(x_1,x_2)$ and $(y_1,y_2)$, there exists $g  \in G$ such that $gx_1 = y_1$ and $g x_2 = y_2$. \\
			Now let $G$ act on an infinite set $\Omega$ and fix $\alpha \in \Omega$. The claim is that the point-stabilizer $M: = \text{Stab}_{G}(\alpha)$ is a maximal subgroup of infinite index. Firstly, since the action is transitive, orbit-stabilizer theorem yields that $|G/M| = | G \alpha | = |\Omega| = \infty$. \\
			Now suppose that $M < L \leq G$. Choose $\ell \in L \setminus M$ and let $\beta: = \ell \alpha \neq \alpha$. Let $\gamma \in \Omega \setminus \{\alpha\}$. Then consider the pairs $(\alpha, \beta)$ and $(\alpha, \gamma)$. By 2-transitivity, there exists $m \in G$ such that $m \alpha = \alpha$ and $m
			\beta = \gamma$ (In general, $M$ acts transitively on $\Omega \setminus \{\alpha\}$). Hence, $m \in M$ with $m \beta = \gamma$. Consequently, 
			\[
			(m \ell) \alpha = m (\ell \alpha) = m \beta = \gamma.
			\]
			Now let $g \in G$ be any element of $G$. Then as $M < L$, $m \ell \in L$ and hence $L$ acts transitively on $\Omega$. Now let $g \in G$. Since $L$ is transitive, there exists $l \in L$ such that $l \alpha = g \alpha$. Thus, we have $l^{-1} g \in M$. Consequently, $g \in L$ and hence $L= G$.

			\item[(ii)] 
			J.O. Button and R. P. Kropholler have shown the following result for a group $G$ as in case $(ii)$. (see \cite{buthell})
			\begin{thm*}
				If the group $G$ is as in Case \textup{(ii)} of the preceding Proposition \ref{prop:osin},
				then $G$ is a generalized Baumslag-Solitar group. Moreover, any generalized
				Baumslag-Solitar group is either $SQ$-universal or it is isomorphic to the
				Baumslag-Solitar group $BS(1,j)$ for some $j\in \mathbb{Z}\setminus\{0\}$,
				or is infinite cyclic.
			\end{thm*} 
			We will follow the proof of this theorem and argue that every possible group arising is either virtually solvable or has a maximal subgroup of infinite index. Kropholler has shown in Theorem C of \cite{Kropholler1990} that a non-cyclic finitely generated group of cohomological dimension 2 that have an infinite cyclic s-normal subgroup are exactly generalized Baumslag-Solitar groups. A non-free one-relator group has cohomological dimension 2 if the relator is not a proper power by Lyndon $\cite{Lynerdon}$. On the other hand, a proper power gives rise to a group with torsion whereas the groups as in case $(ii)$ are torsion-free. \\
			Now let $\Gamma$ be the graph of groups resulting in our generalized Baumslag-Solitar group.
			For every vertex \(v\), let
			\[
			G_v=\langle a_v\rangle\cong \mathbb{Z}.
			\]
			
			If an edge \(e\) joins the vertices \(v\) and \(w\), let its labels at
			\(v\) and \(w\) be \(m_e,n_e\neq 0\), respectively. These labels mean
			that the edge group \(G_e=\langle c_e\rangle\cong\mathbb{Z}\) embeds
			into the adjacent vertex groups according to
			\[
			c_e\longmapsto a_v^{m_e},
			\qquad
			c_e\longmapsto a_w^{n_e}.
			\]
			
			Choose a maximal subtree \(T\subseteq\Gamma\). Then the corresponding
			GBS group has the presentation
			\[
			G=
			\left\langle
			a_v,\
			t_e\
			\ \middle|\
			\begin{aligned}
				a_v^{m_e}&=a_w^{n_e}
				&&\text{if }e=[v,w]\in E(T),\\
				t_ea_v^{m_e}t_e^{-1}&=a_w^{n_e}
				&&\text{if }e=[v,w]\in E(\Gamma)\setminus E(T)
			\end{aligned}
			\right\rangle 
			\]
			where $v \in V(\Gamma)$ and $e \in E(\Gamma) \setminus E(T)$. \\
			We will investigate our group with respect to the number of cycles in $\Gamma$. Without loss of generality, assume that after elementary collapses, every label with $\pm 1$ occurs only on a self-loop.
			
			\begin{enumerate}
				\item[1.] \textbf{$\Gamma$ has more than one cycle.} \\
				For each cycle, we add a stable letter to our group $G$ and we can quotient out $G$ by vertex groups to leave only the stable letters which do not relate to each other. Hence, $G$ surjects onto the free group $F_2$ and thus Margulis-Soifer theorem applies.

				\item[2.] \textbf{$\Gamma$ is a tree.}  \\
				If $\Gamma$ is a tree, then there are no stable letters and 
				\[
				G = \langle a_v \ | \ a_v^{m_e} = a_w^{n_e} \ \text{whenever} \ e = [v,w] \rangle
				\]
				Since $\Gamma$ is a finite graph by the definition of $GBS$, we can find a common power for all the generators $a_v$. Let $z = a_v^{N_v}$ be that common power for some $N_v \neq 0 $. Hence, $z$ commutes with every generator $a_v$ and thus $z$ is central. Now Levitt's Proposition 4.1 in $\cite{levitt}$ states that a $GBS$ group has nontrivial center if and only if it is a mapping torus group $F \rtimes_\alpha \mathbb{Z}$ where $F$ is a nonabelian (otherwise $G$ is solvable) free group and $[\alpha]$ has finite order in $\text{Out}(F)$.
				It is known that $G$ then has a finite-index non-virtually-solvable normal linear subgroup isomorphic to $F \times \mathbb{Z}$ and hence the group is large (see the first part of case $(iii)$ for a proof of this fact).  
				
				\item[3.] \textbf{$\Gamma$ has exactly one cycle which is not a self-loop.}
				
				We remove an edge $e$ whose endpoint vertex groups are $\langle h_1 \rangle$ and $\langle h_2 \rangle$ from the cycle and the remaining graph becomes a tree $T$. Denote the fundamental group of this tree by $H$. Now notice that we can make $G$ into an HNN extension by reattaching the edge $e$, that is, 
				\[
				G = \langle H, t \ | \ th_1^{m}t^{-1} = h_2^{n} \rangle
				\]
				where $m,n \neq 0, \pm 1$. Since the finite tree $T$ has $|E| = | V|-1$ edges, the group $H$ has $|V|$ generators and $|V|-1$ relations. Thus, the number of generators in $G$ will also be one more than that of relators as we add a generator $t$ and a relation $th_1^{m}t^{-1} = h_2^{n}$. Therefore, there is an epimorphism $\theta : H \to \mathbb{Z}$. Moreover, $\theta$ must be injective on every vertex group. Indeed, if $a_v^{m} \in \text{ker}(\theta)$ for some label $m \neq 0$, then $m \theta(a_v) = 0$ which implies that $\theta$ annihilates the vertex group $\langle a_v \rangle$. Now take an edge $e = [v,w]$. We have the relation $a_v^{m} = a_w^{n}$ corresponding to this edge. Since $\theta(a_v)=0$, we also have $\theta(a_w)=0$ and thus $\theta$ annihilates the vertex group $\langle a_w \rangle$. Repeating this process along the connected tree, $\theta$ kills every vertex group; however, the vertex groups generate $H$ which would force $\theta$ to be the zero homomorphism, contradicting the surjectivity. Set $\theta(h_1) = k_1$, $\theta(h_2)=k_2$. Note that since $\theta$ is injective on each vertex group, $k_1 \neq 0$ and $k_2 \neq 0$.  \\ Let $p = k_1 m$ and $q = k_2 n$ and write the Baumslag-Solitar group as $BS(p,q) = \langle a, u \ | \ ua^{p}u^{-1} = a^{q} \rangle$. Now define $\Phi: G \to BS(p,q)$ by $\Phi(h) = a^{\theta(h)}$ and $\Phi(t) = u$. The map is well-defined as it preserves the HNN relation: 
				\[
				\Phi(th_1^{m}t^{-1}) = u \Phi(h_1)^{m} u^{-1} = u a^{k_1 m} u^{-1}
				\]
				and
				\[
				\Phi(h_2^{n}) = a^{k_2 n}
				\]
				which are equal in $BS(p,q)$. Hence, $G$ surjects onto $BS(p,q)$ with $|p|,|q| >1$. 
				If $d: = \text{gcd}(|p|,|q|) > 1$, then impose the relation $a^d=1$. Since $d \ | \ p,q$ the Baumslag-Solitar relation becomes trivial, and we obtain that
				\[
				BS(p,q) \twoheadrightarrow \langle a, u \ | \ a^{d}=1 \rangle \cong \mathbb{Z}_d * \mathbb{Z}.
				\]
				As the free product $\mathbb{Z}_d * \mathbb{Z}$ is non-elementary virtually free, hence it is large. Now pulling back the maximal subgroup of infinite index in $\mathbb{Z}_d * \mathbb{Z}$ twice, we get a maximal subgroup of infinite index of $G$. \\ 
				If $\text{gcd}(|p|,|q|) =1$, then we refer to the Proposition $8.8$ in $\cite{fmms}$: 
				\begin{prop}
					Let $p,q \in \mathbb{Z}\backslash \{0\}$. The following are equivalent:
					\begin{enumerate}
						\item[(a)] $|p| \neq 1$, $|q| \neq 1$, and $|p| \neq |q|$;
						\item[(b)] $\operatorname{BS}(p,q)$ admits a highly transitive and highly faithful action;
						\item[(c)] $\operatorname{BS}(p,q)$ is highly transitive.
					\end{enumerate}
				\end{prop}
				Since $|p|,|q| \geq 2$ and $\text{gcd}(|p|,|q|) =1$, we have $|p| \neq |q|$ and thus $BS(p,q)$ admits a highly transitive action and therefore has a maximal subgroup of infinite index. In fact, we will see in Section 4 that if $|p| \neq |q|$, $|p|,|q|>1$ then $BS(p,q)$ admits a free maximal subgroup of infinite index. 
				
				\item[4.] \textbf{$\Gamma$ has exactly one cycle which is a self-loop} \\
				If the loop $L$ is all of $\Gamma$ then the $G \cong BS(p,q)$ and we have already taken care of these in above paragraph. 
				Suppose now that $L \neq \Gamma$. Choose an edge $e \notin L$. 
				Since $L$ is the only cycle in $\Gamma$, the edge $e$ is separating. 
				Removing the interior of $e$ produces two connected components 
				$\Gamma_1$ and $\Gamma_2$, where $\Gamma_1$ is a tree and 
				$\Gamma_2$ contains $L$. Let
				\[
				G_i=\pi_1(\Gamma_i).
				\]
				Denote the endpoints of $e$ by $v_i\in \Gamma_i$, and let $h_i$ 
				generate the corresponding vertex group. If the labels of $e$ at 
				$v_1$ and $v_2$ are $k$ and $\ell$, respectively, then
				\[
				G \cong G_1
				*_{\langle h_1^k\rangle=\langle h_2^\ell\rangle}
				G_2.
				\]
				Since the labeled graph is reduced and $e$ has distinct endpoints, we have $|k|, |\ell| >1$. 
				
				Because $\Gamma_1$ is a tree, its standard presentation has one more
				generator than relator. Consequently, there is an epimorphism $\theta\colon G_1 \to \mathbb{Z}$.
				
				Choose a prime $p \ | \ k$, and compose $\theta$ with reduction modulo
				$p$ to obtain an epimorphism $\overline{\theta} : G_1 \to \mathbb{Z}_p$.
				
				Since $p \ | \ k$, we have
				\[
				\overline{\theta}(h_1^k)
				=k\overline{\theta}(h_1)
				=0.
				\]
				
				On the other hand, the stable letter corresponding to \(L\) defines an
				epimorphism	$\rho : G_2 \to \mathbb{Z}$
				which sends every vertex group to $0$. In particular $\rho(h_2^\ell)=0$.
				
				Thus, $\overline{\theta}$ and $\rho$ agree on the amalgamated subgroup,
				since both send it to the identity. By the universal property of
				amalgamated free products, they induce an epimorphism $\Phi : G \to \mathbb{Z}_p * \mathbb{Z}$.
				
				The group $\mathbb{Z}_p * \mathbb{Z}$ is a non-elementary virtually free
				group, and hence it has a maximal subgroup of infinite index.
				
			\end{enumerate} 		
			
			\item [(iii)] Finally, we are considering the case of an ascending HNN extension. We will have different treatments for strictly ascending and non-strictly ascending extensions. 
			Let us note that in the proofs below, we will assume that $\text{rank}(F) \geq 2$ because 
			\begin{enumerate}
				\item If $\text{rank}(F)=0$, then $G \cong \mathbb{Z}$
				\item If $\text{rank}(F)=1$ and $\phi$ is an automorphism then either $G \cong \mathbb{Z}^{2}$ or $G$ is the Klein-bottle group.
				\item If $\text{rank}(F)=1$ and $\phi$ is not surjective, then $G \cong BS(1,q)$ which is solvable 	
			\end{enumerate}
			and all these groups are virtually solvable. 
			\item[(a)] \textbf{The extension is non-strictly ascending} \\
			First, suppose that the HNN extension is non-strictly ascending. Then $\phi : F \to F$ is an automorphism, and $G = F \rtimes_\phi \mathbb{Z}$. First assume that $[\phi]$ has a finite order in $Out(F): = Aut(F) /Inn(F)$. Then it is an easy exercise that $G$ has a finite-index non-virtually-solvable linear subgroup isomorphic to $F \times \mathbb{Z}$. Thus $G$ is also a non-virtually-solvable linear group as a finite-index extension. Hence, Margulis-Soifer theorem applies.
            
			
			On the other hand, if $[\phi]$ has infinite order in $Out(F_n)$ then, $G$ is known to be acylindrically hyperbolic by Corollary 1.5 in $\cite{Gen}$. \\
			
			\item[(b)] \textbf{The extension is strictly ascending} \\
			Now let $G$ be a strictly ascending HNN extension of a non-abelian finitely generated free group. We will prove that $G$ has a profinitely dense subgroup and hence has a maximal subgroup of infinite index by the Proposition \ref{prop:profinitedense}. Suppose $\text{rank}(F) = n$ and $P: = \phi (F)$. Notice that $P$ has infinite index in $F$ by the Schreier index formula. Indeed, if $[F:P]=d$ then 
			\[
			n = d(n-1) +1
			\] 
			which implies that $d=1$, contradicting the fact that $\phi$ is non-surjective. Then by the Lemma 5.7 in $\cite{min25}$, there exists $u \in F \setminus \{1\}$ such that $\langle P, u \rangle \cong P * \langle u \rangle$. It is clear that $\langle P, uPu^{-1} \rangle \cong P * uPu^{-1}$. \\
			Now we give an alternative presentation for $G$ to apply Lemma 5.2 from $\cite{min25}$. Let $s = ut$ and $\theta = \text{Inn}(u) \circ \phi$. In other words, we have $\theta(x) = u \phi(x) u^{-1} = u txt^{-1} u^{-1} = sxs^{-1}$. Hence, $G$ has the following presentation 
			\[
			G = \langle F, s \; | \; sxs^{-1} = \theta(x) \; \forall x \in F \rangle.
			\]
			Now applying Lemma 5.2 in $\cite{min25}$ (where $H =P$ and $\phi = \theta$, and $s$ being the stable letter of the extension), we deduce that $\langle P, N, s \rangle \cong P * \langle N, s \rangle$ where $N: = \bigcap_{j=0}^{\infty} \theta^{j}(F)$. In particular, $\langle P, s \rangle \cong P * \langle s \rangle$. \\
			
			Now let $D: = \langle P,s \rangle$. Notice that since $s$ has infinite order, $D \cong F_{n+1}$. We claim that $D$ is a proper profinitely dense subgroup of $G$. It is clear that $D$ is a proper subgroup of $G$.\\

			Now let $\pi : G \to Q$ be a homomorphism to a finite group $Q$. Returning to the presentation $$G = \langle F,t \; | \; txt^{-1} = \phi(x) \rangle,$$ we will show that $\pi(F) = \pi (P)$ and $\pi(t) \in \pi(D)$ so that $D$ and $G$ have the same images. Since $P = \phi(F) = tFt^{-1} \leq F$, we have 
			\[
			\pi(P) = \pi(t) \pi(F)\pi(t)^{-1} \leq \pi(F)
			\]
			and hence the subgroups $\pi(P)$ and $\pi(F)$ are conjugate. Moreover, since $\pi(P)$ and $\pi(F)$ are finite conjugate subgroups, they have the same order. Therefore, $\pi(P) = \pi(F)$ as $\pi(P) \leq \pi(F)$. \\
			Lastly, since $u \in F$ and $\pi(F) = \pi(P) \leq \pi(D)$, we get $\pi(u) \in \pi(D)$. Thus, $\pi(t) = \pi(u)^{-1} \pi(s) \in \pi(D)$. Therefore, $\pi(D) = \pi(G)$.
			
		\end{enumerate}
	\end{proof}

\section{Free maximal subgroups of infinite index}

The main goal of this section is to prove that every non-cyclic one-relator group with torsion and every non-cyclic one-relator $2$-free group has a free maximal subgroup of infinite index. One can see that one-relator groups with torsion might contain both free and non-free maximal subgroups of infinite index simultaneously. For one-relator groups with more than two generators, it is easy to provide examples, for instance, a simple example would be $G_1 = \langle a,b,c \ | \ c^2=1 \rangle$ which surjects onto $\mathbb{F}_2$ after fixing the generators $a,b$ and killing $c$. Then the pullback of any maximal subgroup of infinite index of $\mathbb{F}_2$ will contain an element of order $2$ and hence cannot be free. On the other hand, it must also contain a free maximal subgroup of infinite index by Proposition \ref{prop:torsionand2free} below. 

\medskip 

A more delicate question is whether or not there is a $2$-generated one-relator group containing both free and non-free maximal subgroups of infinite index. One example of such a group is $G_2 = \langle a,b \ | \ [a,b]^{p}=1 \rangle$ where $p$ is a sufficiently large prime number for which a Tarski monster of exponent $p$ exists. Now let $T$ be the Tarski monster group corresponding to $p$ and assume it is generated by $x,y$. Since $T$ has exponent $p$, $[x,y]^{p}=1$ and thus there is an epimorphism $G_2 \to T$ defined by $a \mapsto x$ and $b \mapsto y$. For $T$ is nonabelian, we have $[x,y] \neq 1$. Thus, the subgroup $C = \langle [x,y] \rangle$ has order $p$ and is maximal. However, the preimage of $C$ cannot be free since it will contain $[a,b]$. 

\medskip 

The same phenomenon happens for torsion-free one-relator groups as well. The group $G_3 = \langle a, t \ | \ [a,tat^{-1}]=1 \rangle$ has both free and non-free maximal subgroups of infinite index. To see why there is a free maximal subgroup of infinite index, put $b = tat^{-1}$ and write $H_1 = \langle a,b,t \ | \ [a,b]=1, tat^{-1}=b \rangle$ which makes $H_1$ into a non-ascending HNN extension of $H_2 = \langle a,b \ | \ [a,b]=1 \rangle \cong \mathbb{Z}^{2}$. Then one can obtain a free maximal subgroup of infinite index using Proposition $\ref{highnn}$ (it must be checked that the action on the boundary is topologically free and it is fairly easy to see in this case). To show that it has a non-free maximal subgroup of infinite index, consider the quotient $H_1 / \langle \langle a^2 \rangle \rangle$. Notice that it becomes HNN extension $H_3 = \langle \bar{a}, \bar{b}, \bar{t} \ | \ \bar{a}^{2} = \bar{b}^{2} = [\bar{a}, \bar{b}]=1, \bar{t} \bar{a} \bar{t}^{-1} = \bar{b} \rangle$. Using Proposition $\ref{highnn}$ again we obtain a maximal subgroup of infinite index, whose pullback to $H_1$ cannot be free since it will contain $\langle a^2, b^2 \rangle \cong \mathbb{Z}^{2}$. 

\medskip 

However, not every torsion-free one-relator group possesses a free maximal subgroup of infinite index. In fact, even centerless one-relator torsion-free groups may fail to contain a free maximal subgroup of infinite index. A pleasant example would be $G_4:=BS(n,-n), n\geq 2$ since one can show that every maximal subgroup of infinite index in $BS(n,-n) = \langle a, t \ | \ ta^{n}t^{-1} = a^{-n} \rangle$ contains a copy of $\mathbb{Z}^{2}$. Let $z = a^{n}$ and $N = \langle z \rangle \cong \mathbb{Z}$. Since $a$ commutes with $z$ and conjugation by $t$ inverts $z$, $N \unlhd G_4$. Suppose that $M$ is a maximal subgroup of infinite index. First we observe that $N \leq M$, otherwise $G= MN$ by the maximality of $M$ and hence 
\[
	[G_4: M] = [MN:M] = [N:N \cap M] = \infty
\]
which would force $N \cap M = 1$ since $N \cong \mathbb{Z}$. Consequently, we would have $M < M \langle z^{2} \rangle < G_4$, contradicting the maximality of $M$. Moreover, $M$ contains an element $u$ that inverts $z$. Conjugation on $N$ induces a well-defined surjection $\delta: G_4 \to \mathbb{Z}_2$ defined by $a \mapsto 0$ and $t \mapsto 1$. Now choose an element $u \in M \setminus \ker(\delta)$ (such $u$ exists, by maximality and $[G_4:M]= \infty$). Since $N \unlhd G_4$, conjugation by $u$ is an automorphism of $N$ and hence we have either $uzu^{-1} = z$ or $uzu^{-1} = z^{-1}$. From the repeated application of the relations $aza^{-1} = z$ and $tzt^{-1} = t^{-1}zt = z^{-1}$,  we get $uzu^{-1} = z^{(-1)^{k}}$ where $k$ is the number of occurrences of $t^{\pm 1}$. Then if $uzu^{-1} = z$, $k$ must be even and thus by definition of $\delta$, we would have $u \in \ker(\delta)$. Since $u \notin \ker(\delta)$ by choice, we have $uzu^{-1} = z^{-1}$. Then $u^2$ and $z$ commute. Moreover, if $u^{2k} = z^l$ for some integers $k,l$ then conjugation by $u$ gives $z^{l} = z^{-l}$. Hence, we have $k = l =0$ by torsion-freeness. Thus, $\mathbb{Z}^2 \cong \langle u^2, z \rangle \leq M$.
	
	\begin{prop} \label{prop:torsionand2free}
		If $G$ is a non-cyclic one-relator group that has torsion or is $2$-free, then $G$ has a free maximal subgroup of infinite index.
	\end{prop}

	 First we show that a non-ascending HNN extension acting topologically freely on the boundary of its Bass-Serre tree has a free maximal subgroup of infinite index. 
	
	\begin{prop} \label{highnn}
		Let $\Gamma$ be an HNN extension
		\[
		\Gamma= \operatorname{HNN}(H,A,B)
		\]
		of a countable group $H$ with $A \neq H \neq B$ such that the action of $\Gamma$ on the boundary of its Bass--Serre tree is topologically free. Then $\Gamma$ has a free maximal subgroup of infinite index. 
	\end{prop}

	\begin{proof}
		Start with a free action $H \curvearrowright X$ with infinitely many orbits. Then it extends to a highly transitive action of $\Gamma$ by (proof of) the Theorem C in $\cite{fmms}$. Now fix $x \in X$ and let $M : = \text{Stab}_{\Gamma }(x)$. We know that it is a maximal subgroup of infinite index. \\
		Recall that a group is said to act freely on a graph if there are no edge inversions and vertex stabilizers are trivial. Moreover, Bass-Serre theorem states that a group is free if and only if it acts freely on a tree. We claim that $M$ acts freely on the Bass-Serre tree $T$ of $\Gamma$. \\
		The action $\Gamma$ on its Bass-Serre tree restricts to an action of $M$ on $T$. Since the $\Gamma$-stabilizer of every vertex is conjugate of $H$, for a vertex $aH$, its $M$-stabilizer is then $M \cap aHa^{-1}$. Now suppose that $1 \neq m \in M \cap aHa^{-1}$, that is, the vertex stabilizer is not trivial for some vertex $aH$. Then $m = aha^{-1}$ for some $h \in H$. Since $m$ stabilizes $x \in X$, we have 
		\[
		aha^{-1} x =x.
		\]
		However, this implies that $h(a^{-1}x) = a^{-1}x$ which contradicts the fact that $H$ acts freely on $X$. Hence, $M$ acts freely on a tree and therefore must be free.
	\end{proof}

    \begin{rem}[\textbf{Application to Baumslag-Solitar groups}]
	We note that even though Baumslag-Solitar groups $BS(m,n)$, $|m| \neq |n|$, $|m|,|n|>1$ do not fall into any of the categories in Proposition $\ref{prop:torsionand2free}$, they admit a free maximal subgroup of infinite index. We have already discussed that $BS(n,n)$ and $BS(n,-n)$ do not admit free maximal subgroups of infinite index. In $\cite{harpe}$ (Lemma $20(iii)$), de la Harpe and Préaux have proven that the action of $BS(m,n)$ on the boundary of its Bass-Serre tree is topologically free if and only if $|m| \neq |n|$. Since Proposition $\ref{highnn}$ applies only to non-ascending HNN extensions, we exclude $|m|=1$ and $|n|=1$ (In fact, every maximal subgroup of $BS(1,n)$ has finite index). 
\end{rem}

	\begin{prop} \label{prop:BSfree}
		If a non-cyclic finitely generated one-relator group has torsion or is $2$-free, then it admits a non-ascending HNN splitting whose Bass-Serre boundary action is topologically free. 
	\end{prop}
	
	\begin{proof}
		First suppose that $G$ has torsion. 	By Linton's Theorem 4.13 in $\cite{linton}$, there exists a one-relator hierarchy (see $\cite{linton}$ for precise definitions)
		\[
		X_N \looparrowright \cdots \looparrowright X_1
		\looparrowright X_0 = X.
		\]
		Writing \(G_i = \pi_1(X_i)\), we have \(G_0 = G\) and
		\[
		G_i \cong \operatorname{HNN}(G_{i+1}, A_i, B_i, \psi_i),
		\qquad 0 \leq i < N,
		\]
		where \(A_i, B_i\) are Magnus subgroups, hence free, and \(G_N\)
		is finite cyclic. Since \(G\) is noncyclic, \(N \geq 1\). For simplicity, we will write the HNN extension of $G$ as 
		\[
		G = \text{HNN}(H,A,B,\phi).
		\] Since $G$ has torsion, and $A,B$ are free groups, we have $A \neq H \neq B$.\\ Corollary 6.7 in the same paper states that every such hierarchy is $\mathbb{Z}$-stable. We know that a one-relator group with torsion is hyperbolic. Since hyperbolic groups do not contain Baumslag-Solitar groups, the Theorem 7.1 $(3) \Rightarrow (2)$ in $\cite{linton}$ implies that the above one-relator hierarchy is acylindrical. In particular, $G$ acts on its Bass-Serre tree acylindrically, that is, there exists an integer $L \geq 1$ such that every geodesic segment $I \subseteq T$ of length at least $L$ has finite pointwise stabilizer $G_{(I)}$.  \\
		For any edge $e \in I$, we have $G_{(I)} \leq G_e$. The edge stabilizers are conjugates of a free group, and thus they are torsion-free. Consequently, $G_{(I)}$ must be trivial as a finite torsion-free group. \\
		In the Bass-Serre tree of an HNN extension, every vertex has degree $[H:A]+ [H:B] \geq 4$ since both subgroups are proper. Now suppose that $g \in G$ fixes a nonempty open set $U$ of $\partial T$. By definition of the boundary topology, there is a half-tree $C$ such that $\emptyset \neq \partial C \subseteq U$. Suppose that $C$ is the component containing $v$ after deleting an edge joining some $u$ and $v$. Since each vertex has degree of at least $4$, we have $3$ distinct edges starting at $v$. Choose $3$ distinct ends $\xi_1, \xi_2$, and $\xi_3$. Their tripod center is the vertex $v$ and is fixed since $g \in G$ fixes each bi-infinite geodesic line connecting the ends $\xi_1,\xi_2$ and $\xi_3$ (note that only two ends would not suffice, since the geodesic could be fixed by translation which would not preserve $v$). In particular, the unique ray from $v$ towards $\xi_1$ is preserved by $g \in G$. Since $g$ preserves distances from $v$ and $v$ is fixed, it fixes the ray pointwise. In particular, $g \in G$ fixes pointwise a segment $I$ of length greater than $L$ and thus $g \in G_{(I)}$. However, we have shown above that all such $G_{(I)}$ are trivial. Hence, the result follows.  \\
		
		Now assume that $G$ is non-cyclic $2$-free one-relator group. Since $G$ is non-cyclic $2$-free, its primitivity rank is at least 3 and thus Theorem 1.3 and Reamrk 1.7 in $\cite{louder}$ implies that $G$ has negative immersions. Now choose a one-relator hierarchy
		\[
		X_N \looparrowright \cdots \looparrowright X_1 \looparrowright X_0 = X.
		\] 
		Each $X_i$ inherits negative immersions since its immersion into $X$ can be composed with any immersion into $X_i$. Thus, every splitting in the hierarchy is $\mathbb{Z}$-stable by Corollary 6.9 in $\cite{linton}$. Clearly, a $2$-free group is torsion-free and cannot contain Baumslag-Solitar groups. Hence, Theorem 7.1 $(3) \Rightarrow (2)$ again implies that $G$ acts acylindrically on its Bass-Serre tree, that is, there exists $L \in \mathbb{N}$ such that the pointwise stabilizer group $G_{(I)}$ of any segment longer than $L$ is finite. \\
		Finally, we observe that the HNN extension is non-ascending. Since $G$ is torsion-free, the pointwise segment stabilizer coming from the acylindricity of the action must be trivial. Suppose that $A = H$. Then from the relation $tat^{-1} = \phi(a)$, we obtain that $tHt^{-1}=B \leq H$ and consequently, $H \geq tHt^{-1} \geq \cdots \geq t^{n}Ht^{-n}$. Considering the geometric segment $I_n$ with the successive vertices $H,tH,...,t^{n}H$, its pointwise stabilizer is $G_{(I_n)} = \bigcap_{i=0}^{n} t^{i} H t^{-i} = t^{n}Ht^{-n}$ which is infinite, contradicting $G_{(I)} =1$ whenever $|I| \geq L$. Thus, the extension cannot be ascending. 
		
	\end{proof}

	\begin{proof}[$\textbf{Proof of Proposition \ref{prop:torsionand2free}}$]
		Combining Proposition \ref{highnn} and Proposition \ref{prop:BSfree} , the result follows.
	\end{proof}

     Proposition \ref{prop:torsionand2free} proves part a) of Theorem \ref{thm:main2}. We have already provided examples in this section to show parts b) and c). These conclude the proof of the theorem. 

\medskip 

 \section{Frattini subgroup of one-relator groups}

   The Baumslag-Solitar group, despite having only finite index maximal subgroups, is still rich with maximal subgroups (in particular, it has non-conjugate maximal subgroups), and the Frattini subgroup is trivial. Same holds for all other infinite solvable one-relator groups (i.e. $\mathbb{Z}, \mathbb{Z}^2, \pi_1(\mathrm{Klein \ bottle})$). In this section, we will show that non-solvable one-relator groups also have a trivial Frattini subgroup.  

\medskip 

     Let $ G=\langle x_1,\ldots,x_k\mid R\rangle $ be a one-relator presentation of $G$. Since $G$ is non-solvable, $k\geq 2$.

\medskip 

We first recall a theorem of B. B. Newman \cite{Newman1968} that a one-relator group has a trivial Frattini subgroup if either it has torsion and has more than one generator, or it is torsion-free and has more than two generators. Consequently, if $G$ has torsion, then $ \Phi(G)=1,$ and if $G$ is torsion-free and $k\geq 3$, again $ \Phi(G)=1.$ Thus, it remains to consider the case $G=\langle a,b\mid R\rangle, $ where $G$ is torsion-free and non-solvable.

\medskip 

We shall use the following consequence of the argument of Gelander and Glasner  \cite{ggc}. Let $\Psi(L)$ denote the intersection of all maximal subgroups of infinite index of a group $L$. Clearly, $\Phi(L)\leq \Psi(L). $ If a countable group $L$ acts
minimally on a tree $T$, $|\partial T|>2$, and the action fixes no end, then $\Psi(L)$ acts trivially on $T$. Hence $\Phi(L)\leq \ker(L\curvearrowright T). \ (1) $

\medskip 

We now use the result of Osin-Minasyan stated as Proposition 6.1 in \cite{min25}. It implies that a two-generated one-relator group $G$ satisfies at least one of the following:

\begin{enumerate}
\item[(i)] $G$ is a generalized Baumslag--Solitar group;

\item[(ii)] $G$ is a strictly ascending HNN extension of a finitely
generated non-abelian free group;

\item[(iii)] $G$ is acylindrically hyperbolic.
\end{enumerate}

We consider these three cases separately.

\medskip

\noindent
{\bf Case 1. $G$ is acylindrically hyperbolic.}

Hull proved that the Frattini subgroup of every countable
acylindrically hyperbolic group is finite \cite{hullsmall}. Since $G$ is torsion-free, $G$ has no non-trivial finite subgroup. Therefore $ \Phi(G)=1. $

\medskip

\noindent
{\bf Case 2. $G$ is a strictly ascending HNN extension.}

In this case $G=\langle F,t\mid tft^{-1}=\varphi(f),\ f\in F\rangle, $
where $F$ is a free group of finite rank $r\geq 2$, and
$\varphi:F\longrightarrow F$ is injective but not surjective. We first observe that $[F:\varphi(F)]=\infty.\ (2)$ Indeed, $\varphi(F)\cong F$, so $\operatorname{rank}\varphi(F)=r.$ If $[F:\varphi(F)]=d<\infty$, then Schreier's formula gives $r=1+d(r-1).$ Since $r\geq 2$, this forces $d=1$, contradicting the fact that $\varphi(F)$ is a proper subgroup of $F$.




\medskip

Suppose, for a contradiction, that $\Phi(G)\neq1 $. Let $ \chi:G\to \mathbb Z, \chi(t)=1, \chi(F)=0.$ Since epimorphisms send Frattini subgroups into Frattini subgroups, $ \Phi(G)\leq\ker\chi =\displaystyle \mathop{\bigcup}_{n\geq0}t^{-n}Ft^n.$ Write  $P = \phi (F)$. Normality of $\Phi(G)$ therefore implies that $\Phi(G)\cap P\neq 1 $ (conjugate a nontrivial Frattini element into $F$, and then conjugate once more by $t$). Choose  $1\neq w\in\Phi(G)\cap P$. By Lemma 5.7 of \cite{min25}, choose $v\in F$ such that
$ \langle P,v\rangle\cong P*\langle v\rangle. $ Let $u=vwv^{-1}$. Then $u\in\Phi(G)$, and free-product normal forms give

$$ \langle P,u\rangle\cong P*\langle u\rangle, \ \mathrm{and} \ \langle P,uPu^{-1}\rangle\cong P*(uPu^{-1}). $$

Here the first assertion uses $1\neq w\in P$: the cyclic group generated by $vwv^{-1}$ lies in the conjugate free factor $vPv^{-1}$.

\medskip 

Now let $s=ut, \theta=\operatorname{Inn}(u)\circ\varphi $.  Then $G=\langle F,s\mid sfs^{-1}=\theta(f),\ f\in F\rangle,  \theta(F)=uPu^{-1}. $

\medskip 

Applying Lemma 5.2 of \cite{min25} (as in Section 3), with $H=P$, gives

$$ D:=\langle P,s\rangle\cong P*\langle s\rangle\cong F_{r+1}. $$

\medskip 

Since $r+1\geq3$, whereas $G$ is two-generated, $D\neq G$. On the other hand, $\langle D,u\rangle=G $ because the left-hand side contains $t=u^{-1}s$, and then contains $F=t^{-1}Pt$. This contradicts $u\in\Phi(G)$, since adjoining a Frattini element to a proper subgroup cannot generate $G$. Thus $\Phi(G)=1.$

\noindent
{\bf Case 3. $G$ is a generalized Baumslag-Solitar group.}

Let $T$ be a minimal GBS tree for $G$. Since $G$ is non-solvable, the
action is non-elementary and fixes no end. In particular, $|\partial T|>2.$
Hence Gelander-Glasner's theorem gives
$\Phi(G)\leq K, \ K=\ker(G\curvearrowright T). $ (3)

\medskip 

Let $\Delta:G\longrightarrow \mathbb{Q}^{*} $ be the modular homomorphism of $G$. Suppose first that $G$ is not unimodular; that is, $\Delta(G)\not\subseteq\{\pm1\}.$
For a non-elementary GBS group, Levitt proved that the kernel of the
action on the Bass-Serre tree can be non-trivial only in the unimodular case \cite{levitt2007}. Thus $K=1.$
Equation (3) therefore gives $\Phi(G)=1.$

\medskip 

It remains to consider the unimodular case. Thus
$\Delta(G)\subseteq\{\pm1\}. $ Levitt constructs in this case a homomorphism
$\tau:G\longrightarrow\operatorname{Isom}(\mathbb R) $
such that
$$
\tau(G)\cong
\begin{cases}
\mathbb Z, & \Delta(G)=\{1\},\\
D_{\infty}, & \Delta(G)=\{\pm1\}, \end{cases} \  $$
and such that
$ \ker\tau \ 
\text{contains no non-trivial elliptic element of }G.
\  $

\medskip 

We use the elementary fact that if
$q:L\twoheadrightarrow Q $ is an epimorphism, then $ q(\Phi(L))\leq\Phi(Q). $
Now $\Phi(\mathbb Z)=1 \  \text{and} \  \Phi(D_{\infty})=1. \ $
Using the epimorphism  $\tau:G\twoheadrightarrow\tau(G)$, we obtain $\tau(\Phi(G))=1.$ Therefore $\Phi(G)\leq\ker\tau. $

\medskip 

On the other hand, by (3), $\Phi(G)\leq K.$ Thus, every element of $\Phi(G)$ acts trivially on the Bass-Serre tree $T$, and hence every element of $\Phi(G)$ is elliptic. But $\ker\tau$ contains no non-trivial elliptic element. Thus, $\Phi(G)=1.$

\medskip 

This settles all three alternatives. Together with Newman's theorem,
we conclude that $\Phi(G)=1 $ for every non-solvable one-relator group $G$.
	
	\medskip 

    \section{Questions}

     We would like to ask several questions. We believe that these questions are within reach.

 \medskip 
 
     {\bf Question 1.} Is there a one-relator group with a finitely generated infinite index maximal subgroup? 

     For a broad class of one-relator groups, we can show the non-existence of such a maximal subgroup; however, we do not know a single example of a one-relator group which possesses a finitely generated maximal subgroup of infinite index.   

    \medskip 
 
     {\bf Question 2.} Can one characterize one-relator groups where every infinite index maximal subgroup is free? 

      The question seems to be interesting for linear groups as well. It is related to two well-known questions. The first one is a conjecture due to Mikhael Gromov: {\em Does every one-ended hyperbolic group contain a subgroup isomorphic to a closed hyperbolic surface group?} The second question (due to Kevin Whyte) asks: {\em if $\Gamma $ is a one-ended hyperbolic group which is not a virtually surface group, can all infinite index subgroups of $\Gamma $ be free?} Our question seems to be  easier. Let us also mention that Wilton \cite{wilton2026surfacegroupscubulatedhyperbolic} has recently shown that if every infinite index subgroup of an infinite one-relator group $\Gamma $ is free, then $\Gamma $ is a closed surface group or a free group.         

     \medskip 
 
     {\bf Question 3.} Can one identify all one-relator groups $\Gamma $ that have a maximal subgroup $M$ such that $gMg^{-1}\cap M = \{1\}$ for some $g\in \Gamma $?

     \medskip 

      The answer is negative for one-relator groups with an infinite center. Gelander and Meiri establish this result for $PSL(n,\mathbb{Z}), n\geq 3$ \cite{gm}, i.e. these groups possess a maximal subgroup having a trivial intersection with its conjugate. 

      \medskip

     {\bf Question 4.} Do all non-solvable one-relator groups have $2^{\aleph _0}$ maximal subgroups? 

     Again, in \cite{gm}, a positive answer is given for $SL(n,\mathbb{Z}), n\geq 3$. In the original works of Margulis and Soifer, it is shown that finitely generated non-virtually solvable linear groups have uncountably many infinite index maximal subgroups.

		\bibliographystyle{IEEEtranS}
		\bibliography{refs}
		
	\end{document}